\documentclass[11pt,a4paper,reqno]{amsart}
\usepackage[T1]{fontenc}
\usepackage{lmodern}
\usepackage[a4paper,margin=25mm]{geometry}
\usepackage{amsmath,amssymb,mathtools,mathrsfs}
\usepackage{microtype,needspace,etoolbox}
\usepackage[hidelinks]{hyperref}
\numberwithin{equation}{section}
\newtheorem{maintheorem}{Theorem}

\newtheorem{lemma}{Lemma}[section]
\newtheorem{obataidentity}{Theorem}

\newtheorem{proposition}[lemma]{Proposition}
\theoremstyle{remark}

\DeclareMathOperator{\Ric}{Ric}
\DeclareMathOperator{\tr}{tr}
\DeclareMathOperator{\Div}{div}
\newcommand{\ip}[2]{\langle #1,#2\rangle}
\newcommand{\dv}{\,\mathrm{d}v_g}
\newcommand{\dvo}{\,\mathrm{d}v_{g_0}}
\newcommand{\Dw}{\Div_{u,g}}
\newcommand{\Ac}{\mathcal A}
\newcommand{\uloggrad}{\nabla_g\log u}

\allowdisplaybreaks[4]
\hypersetup{pdftitle={Rigidity of conformal metrics with constant sixth-order Q-curvature},pdfauthor={Daowen Lin and Wangzhe Wu}}
\title[Sixth-order $Q$-curvature rigidity]{Rigidity of conformal metrics with constant sixth-order $Q$-curvature}
\author{Daowen Lin}
\address[Daowen Lin]{Geometric Partial Differential Equations Unit, Okinawa Institute of
Science and Technology Graduate University, Okinawa, Japan}
\email{daowen.lin@oist.jp}
\author{Wangzhe Wu}
\address[Wangzhe Wu]{School of Mathematics and Statistics, Ningbo University, Ningbo,
Zhejiang, People's Republic of China}
\email{wuwz18@mail.ustc.edu.cn}
\date{}
\newcommand{\manuscriptspacing}{%
  \setlength{\abovedisplayskip}{5pt plus 2pt minus 1.5pt}%
  \setlength{\belowdisplayskip}{5pt plus 2pt minus 1.5pt}%
  \setlength{\abovedisplayshortskip}{2pt plus 1pt minus 1pt}%
  \setlength{\belowdisplayshortskip}{3pt plus 1pt minus 1pt}%
  \setlength{\jot}{1.8pt}%
  \setlength{\parskip}{0pt}%
  \predisplaypenalty=10000
  \postdisplaypenalty=0
}
\apptocmd{\normalsize}{\manuscriptspacing}{}{}
\AtBeginDocument{\manuscriptspacing}
\makeatletter
\patchcmd{\@maketitle}{\dimen@34\p@}{\dimen@24\p@}{}{}
\patchcmd{\@setauthors}{\@topsep30\p@}{\@topsep24\p@}{}{}
\patchcmd{\section}{\@plus\linespacing}{\@plus2\p@}{}{}
\patchcmd{\subsection}{\@plus.7\linespacing}{\@plus1\p@}{}{}
\makeatother
\begin{document}
\begin{abstract}
Let $(M^n,g_0)$ be a closed connected Einstein manifold with positive scalar
curvature, where $n\ge6$. Using an Obata-type identity, we prove that every
smooth metric conformal to $g_0$ with constant sixth-order $Q$-curvature
is Einstein.
\end{abstract}
\maketitle

\section{Introduction}

Let $(M^n,g_0)$ be a closed Einstein manifold with positive scalar curvature.
Obata~\cite{Obata1971} proved that every metric conformal to $g_0$ with
constant scalar curvature is Einstein. In particular, the conformal factor
is constant unless $(M,g_0)$ is a round sphere up to scale.

The conformally covariant powers of the Laplacian constructed by
Graham--Jenne--Mason--Sparling~\cite{GJMS} and their associated
$Q$-curvatures lead to higher-order versions of this problem.
An earlier form of the fourth-order Obata-type integral identity already
appeared in Gursky's 1997 work~\cite{Gursky1997} on the round four-sphere;
see Li--Wei~\cite{LiWei} for this connection.
V\'etois~\cite{Vetois} subsequently established rigidity for constant
fourth-order $Q$-curvature in positive Einstein conformal classes,
combining an Obata-type identity with a scalar-curvature positivity
argument based on Gursky--Malchiodi~\cite{GurskyMalchiodi}.
Case~\cite{Case} gave a geometric formulation of this approach and extended
it to mixed fourth-order curvatures; see also Li--Wei~\cite{LiWei}
for related identities and applications.

The extension of the Obata--V\'etois uniqueness theorem to higher-order
$Q$-curvatures is explicitly identified as an open problem by
Case and Gover~\cite[Section~4.2]{CaseGover}.
Theorem~A resolves its sixth-order case on closed positive Einstein
manifolds in dimensions $n\ge 6$.
The result requires no local conformal flatness assumption, and positivity
of the scalar curvature of the target metric is a consequence of the equation.

We write $R_g$ for the scalar curvature and $A_g$ for the Schouten tensor.
The fourth- and sixth-order curvatures are denoted by $Q_{4,g}$ and $Q_{6,g}$,
respectively. Our normalization is
\begin{equation}\label{eq:normalization}
 P_{6,g}(1)=\frac{n-6}{2}Q_{6,g},\qquad
 \Delta_g f=\tr_g(\nabla_g^2f).
\end{equation}
The local formula and the Einstein factorization of $P_{6,g_0}$ are taken
from Chen--Hou~\cite{ChenHou}. Our main result is the following.

\begin{maintheorem}\label{thm:main}
Let $(M^n,g_0)$ be a smooth closed connected Einstein manifold, where
$n\ge6$ and $R_{g_0}>0$. Suppose that $g=u^2g_0$, with $u$ smooth and
positive, has constant sixth-order $Q$-curvature $Q_{6,g}=\Lambda$.
Then $g$ is Einstein and $R_g>0$. Furthermore, $g=a^2g_0$ for a constant
$a>0$, unless $(M,g_0)$ is a round sphere up to scale. In the round case,
$g=a^2\phi^*g_0$ for a conformal diffeomorphism $\phi$. In either case,
\begin{equation}\label{eq:Einstein-Q6}
 \Lambda=\frac{(n^2-4)(n^2-16)}{32n^2(n-1)^3}R_g^3>0.
\end{equation}
\end{maintheorem}

In the round conformal class, the conclusion follows from the classification
of entire solutions by Wei--Xu~\cite{WeiXu} when $n>6$ and from the spherical
uniqueness theorem of Chang--Yang~\cite{ChangYang} when $n=6$.
In dimension six, finite volume on $\mathbb R^6$ alone does not force an
entire solution to be spherical; see
Martinazzi~\cite{MartinazziClassification,MartinazziVolume} and
Hyder--Martinazzi~\cite{HyderMartinazzi}. Smooth compactification is
therefore essential when passing from the Euclidean equation to the
closed-manifold problem.

The passage from fourth to sixth order introduces higher derivatives of
the scalar curvature, cubic curvature terms, and contractions with the
Bach tensor. These terms have no manifest sign, so obtaining an
Obata-type identity with a coercive remainder is a central difficulty.
On a general Einstein background, the Weyl curvature must also be retained.
We combine a weighted Green identity with a Bach-corrected Newton tensor
to group the highest-order terms and cancel the cubic and Bach
contractions. A further tensor correction yields a remainder controlling
the trace-free Schouten tensor. The positivity needed for this estimate
is derived from the sixth-order equation: the Einstein factorization and
an auxiliary conformal metric imply positivity of the scalar curvature
and of the function $\Theta_u$ defined in \eqref{eq:Theta}.
In the critical dimension six, the logarithmic curvature transformation
law is treated by a separate conformal path. Together, these ingredients
yield an integral identity forcing the trace-free Schouten tensor to vanish.

Section~\ref{sec:obata-first} states the identity and deduces Theorem~A.
Sections~\ref{sec:conformal} and~\ref{sec:positivity} establish the conformal
formulas and positivity. Sections~\ref{sec:construction}
and~\ref{sec:uniform-closure} construct the vector field $X$ and prove
the estimate for the scalar function $F$.

\section{The Obata identity and its consequence}\label{sec:obata-first}

Throughout the paper, $g_0$ is the background metric and $g=u^2g_0$ is the
target metric. We write $\nabla_g$, $\Delta_g$, and $\Div_g$ for the connection,
Laplacian, and divergence of $g$; operators for other metrics carry their
metric subscript explicitly. Unless another metric is indicated, all norms,
inner products, traces, and index contractions are taken with respect to $g$.
In indexed formulas, $\nabla_i=(\nabla_g)_i$,
$u_i=\nabla_i u$, and $u_{ij}=\nabla_i\nabla_j u$; indices are raised using
$g$, and repeated indices are summed.

For a symmetric covariant two-tensor $T$, we use $T(v)$ for the vector determined by
$\langle T(v),w\rangle=T(v,w)$. Thus $T^2$ and $T^3$ denote composition
through $g^{-1}T$. We identify the one-form $\Div_gT$, whose components are
$\nabla^iT_{ij}$, with a vector using $g$.

After a constant rescaling of the background and the reciprocal rescaling
of $u$, we may assume
\begin{equation}\label{eq:background}
 \Ric_{g_0}=(n-1)g_0,\qquad g=u^2g_0.
\end{equation}
Set
\begin{equation}\label{eq:Schouten}
 A_g=\frac1{n-2}\left(\Ric_g-\frac{R_g}{2(n-1)}g\right),\qquad
 J=\tr_g A_g=\frac{R_g}{2(n-1)},\qquad
 E=A_g-\frac Jn g.
\end{equation}
The formulas proved in Lemma~\ref{lem:conformal} also give
\begin{align}
E&=-u^{-1}\left(\nabla_g^2u-\frac{\Delta_gu}{n}g\right),\notag\\
J&=\frac{n(1+|\nabla_g u|^2)}{2u^2}-\frac{\Delta_gu}{u}.
\label{eq:JS-from-u}
\end{align}
For a vector field $V$, define
\begin{equation}\label{eq:weighted-div}
 \Dw V=u^{-1}\Div_g(uV)=\Div_gV+\ip{\nabla_g\log u}{V}.
\end{equation}
Thus $\int_M u\Dw V\dv=0$ on a closed manifold.

\subsection{The sixth-order identity}
On a closed manifold, the classical Obata identity obtained from Case's identity with $T=E$ is
\begin{equation}\label{eq:classical-obata-integral}
 \int_M u|E|^2\dv
 =\frac{n-1}{n}\int_M\ip{\nabla_gu}{\nabla_gJ}\dv.
\end{equation}
For the sixth-order equation, the corresponding source terms involve
$Q_{6,g}$ and its gradient. The normalized background value is
\begin{equation}\label{eq:reference}
 Q_{6,g_0}=\frac{n(n^2-4)(n^2-16)}{32}.
\end{equation}
On the set $u^2J+2+2|\nabla_g u|^2>0$, define
\begin{align}
\Theta_u={}&-\Delta_gJ-4|E|^2+2\ip{\uloggrad}{\nabla_gJ}\notag\\
&+\frac{n-4}{2nu^2}J\left[n(1+|\nabla_g u|^2)+(n+2)u^2J\right]\notag\\
&-\frac{2u^2}{(n-4)(u^2J+2+2|\nabla_g u|^2)}
 \left|\nabla_gJ-4E(\nabla_g\log u)
                   +\frac{n-4}{n}J\nabla_g\log u\right|^2.
\label{eq:Theta}
\end{align}
By \eqref{eq:Theta-geometry} at $t=1$, $\Theta_u$ is a positive multiple
of the scalar curvature of the auxiliary metric $\widehat g_1$ defined
in \eqref{eq:auxiliary}.

\begin{obataidentity}[Sixth-order Obata identity]\label{thm:obata}\label{u:prop:coercivity}
Let $n\ge6$ and assume \eqref{eq:background}.
For every smooth positive $u$, the vector field $X$ and scalar function
$F$ constructed in \eqref{u:eq:vector-field}--\eqref{u:eq:scalar-function} satisfy
\begin{equation}\label{u:eq:source}
 \Dw X=F-\ip{\uloggrad}{\nabla_gQ_{6,g}}-\frac{n(n-4)}{8(n-1)} \left(J-\frac n2\right)(Q_{6,g}-Q_{6,g_0}).
\end{equation}
At every point where $J>0$ and $\Theta_u\ge0$,
\begin{equation}\label{u:eq:coercivity}
 \frac8nF\ge12|E|^2\Theta_u+\frac{|\nabla_gJ|^2}{u^2}.
\end{equation}
On a closed manifold, the integral form of the identity is
\begin{align}
\int_MuF\dv
 ={}&\int_M\ip{\nabla_gu}{\nabla_gQ_{6,g}}\dv\notag\\
&+\frac{n(n-4)}{8(n-1)}\int_Mu\left(J-\frac n2\right)
                       (Q_{6,g}-Q_{6,g_0})\dv.
\label{eq:obata-integral}
\end{align}
\end{obataidentity}
The stronger estimate \eqref{u:eq:collection} also retains a nonnegative
scalar square. The construction of the vector field and the proof of
Theorem~\ref{thm:obata} occupy Sections~\ref{sec:construction} and
\ref{sec:uniform-closure}.

\subsection{Deduction of the rigidity theorem}
We first deduce Theorem~A from Theorem~\ref{thm:obata} and the positivity
result in Proposition~\ref{prop:J-positive}, proved in Section~\ref{sec:positivity}.
\begin{proof}[Proof of Theorem~A]
Normalize the background by \eqref{eq:background}.
Lemma~\ref{lem:constant-Q-positive} gives $\Lambda>0$.
Since $Q_{6,a^2g}=a^{-6}Q_{6,g}$, the rescaling
\begin{align}
g&\longmapsto
 \left(\frac{32\Lambda}{n(n^2-4)(n^2-16)}\right)^{\frac{1}{3}}g,\notag\\
u&\longmapsto
 \left(\frac{32\Lambda}{n(n^2-4)(n^2-16)}\right)^{\frac{1}{6}}u
\label{eq:target-normalization}
\end{align}
makes $Q_{6,g}=Q_{6,g_0}$. We retain the notation $g,u$.

Proposition~\ref{prop:J-positive} gives $J,\Theta_u>0$.
Since $Q_{6,g}=Q_{6,g_0}$ is constant,
\eqref{eq:obata-integral} and \eqref{u:eq:coercivity} yield
\begin{equation}\label{eq:all-dimensional-coercivity}
 0=\frac8n\int_MuF\dv\ge12\int_Mu|E|^2\Theta_u\dv.
\end{equation}
Thus $E=0$, so $g$ is Einstein with positive scalar curvature.
In particular, $J$ is constant. Equations \eqref{eq:Q6-div} and
\eqref{eq:reference} then give $J=\frac n2$ and \eqref{eq:Einstein-Q6}.
The latter equality is preserved by constant rescaling.

Finally, Obata's conformal rigidity theorem~\cite{Obata1971}
implies that $g=a^2g_0$ for some $a>0$, unless $(M,g_0)$ is a round
sphere up to scale. In the round case, $g=a^2\phi^*g_0$ for a conformal
diffeomorphism $\phi$. Undoing \eqref{eq:target-normalization} only
changes the homothety constant.
\end{proof}

\section{Conformal preliminaries}\label{sec:conformal}

The fourth-order curvature is
\begin{equation}\label{eq:Q4}
 Q_{4,g}=-\Delta_gJ-2|A_g|^2+\frac n2J^2
     =-\Delta_gJ-2|E|^2+\frac{n^2-4}{2n}J^2.
\end{equation}

\begin{lemma}\label{lem:conformal}
Under \eqref{eq:background}, the formulas \eqref{eq:JS-from-u} hold, and
\begin{gather}
\nabla_g^2u=-uE+\frac{\Delta_gu}{n}g,\notag\\
\Delta_gu=\frac{n}{2u}(1+|\nabla_g u|^2)-uJ.
\label{eq:scale-Hessian}
\end{gather}
\begin{gather}
\nabla_i\nabla_j(\log u)=-E_{ij}+\frac{\Delta_gu}{nu}g_{ij}
              -\frac{(\nabla_i u)(\nabla_j u)}{u^2},\notag\\
\Div_gE=\frac{n-1}{n}\nabla_gJ.
\label{eq:log-Hessian}
\end{gather}
\begin{equation}\label{eq:u-scale-gradient}
\nabla_g\left(\frac{1+|\nabla_g u|^2}{u^2}\right)
 =-2E(\nabla_g\log u)
 -\left(\frac{2J}{n}+\frac{1+|\nabla_g u|^2}{u^2}\right)\nabla_g\log u.
\end{equation}
\begin{align}
&\Delta_g\left(\frac{1+|\nabla_g u|^2}{u^2}\right)
 +2\left\langle\nabla_g\log u,
       \nabla_g\left(\frac{1+|\nabla_g u|^2}{u^2}\right)\right\rangle\notag\\
&\qquad+2\ip{\nabla_g\log u}{\nabla_gJ}
 =2|E|^2+\frac2n J^2-\frac{n(1+|\nabla_g u|^2)^2}{2u^4}.
\label{eq:u-scale-laplacian}
\end{align}
Consequently,
\begin{gather}
\Dw(\nabla_g\log u)=\frac{\Delta_gu}{u},\notag\\
\Dw\bigl(E(\nabla_g\log u)\bigr)
   =\frac{n-1}{n}\ip{\nabla_g\log u}{\nabla_gJ}-|E|^2.
\label{eq:weighted-basic}
\end{gather}
\end{lemma}
\begin{proof}
For a smooth function $\omega$, the Schouten transformation law is
\begin{equation}\label{eq:Schouten-transform}
 A_{e^{2\omega}g}=A_g-\nabla_g^2\omega+d\omega\otimes d\omega
                  -\frac12|\nabla_g\omega|^2g.
\end{equation}
The Einstein condition \eqref{eq:background} gives $A_{g_0}=\frac12g_0$.
Taking $\omega=-\log u$ in \eqref{eq:Schouten-transform} therefore gives
\begin{equation}\label{eq:background-Schouten}
 \frac12u^{-2}g=A_g+u^{-1}\nabla_g^2u
                       -\frac12u^{-2}|\nabla_gu|^2g.
\end{equation}
Its trace and trace-free part, with \eqref{eq:Schouten}, give
\eqref{eq:scale-Hessian} and \eqref{eq:JS-from-u}. Dividing the first
equation in \eqref{eq:scale-Hessian} by $u$ gives the Hessian formula in
\eqref{eq:log-Hessian}. The contracted Bianchi identity applied to
\eqref{eq:Schouten} gives $\Div_gA_g=\nabla_gJ$, hence the divergence
formula in \eqref{eq:log-Hessian}.

The product rule and \eqref{eq:weighted-div}, with
\eqref{eq:log-Hessian} and $\tr_gE=0$, prove \eqref{eq:weighted-basic}.
Differentiating $\frac{1+|\nabla_gu|^2}{u^2}$ and substituting
\eqref{eq:scale-Hessian} gives \eqref{eq:u-scale-gradient}.
Taking its divergence and using \eqref{eq:log-Hessian},
\eqref{eq:weighted-basic}, and the second equation in
\eqref{eq:scale-Hessian} gives \eqref{eq:u-scale-laplacian}.
For later use, the transformation of the Laplacian is
\begin{equation}\label{eq:Laplacian-transform}
 \Delta_{u^2g_0}f=u^{-2}\left[\Delta_{g_0}f+
        (n-2)\ip{\nabla_{g_0}\log u}{\nabla_{g_0}f}_{g_0}\right].
\end{equation}
\end{proof}

\subsection{Case's Einstein-scale identity}
Case's identity~\cite[(2.5)]{Case} takes the following form.
\begin{lemma}\label{lem:case-scale}
For every symmetric two-tensor $T$,
\begin{equation}\label{eq:case-tensor-general}
\Dw\bigl(T(\nabla_g\log u)\bigr)
 =\ip{\Div_gT}{\nabla_g\log u}-\ip{T}{E}
                   +\frac{\Delta_gu}{nu}\tr_gT.
\end{equation}
Consequently, on a closed manifold,
\begin{equation}\label{eq:case-integrated}
0=\int_M u\left(
 \ip{\Div_gT}{\nabla_g\log u}-\ip{T}{E}
              +\frac{\Delta_gu}{nu}\tr_gT\right)\dv.
\end{equation}
\end{lemma}
\begin{proof}
The product rule gives
\[
 \Dw\bigl(T(\nabla_g\log u)\bigr)=\ip{\Div_gT}{\nabla_g\log u}+\ip T{\nabla_g^2\log u}+T(\nabla_g\log u,\nabla_g\log u).
\]
Substitute \eqref{eq:log-Hessian}. The two quadratic terms in $\nabla\log u$ cancel, proving
\eqref{eq:case-tensor-general}. Multiplication by $u$ and integration give
\eqref{eq:case-integrated}. Taking $T=E$ and using
$\Div_gE=\frac{n-1}{n}\nabla_gJ$ gives \eqref{eq:classical-obata-integral}.
\end{proof}

\subsection{An algebraic trace-free tensor}
For any vector $V$, put
\begin{equation}\label{eq:case-AS}
 \Ac_E(V)_{ijk}=E_{ij}V_k-E_{ik}V_j+
       \frac{(E(V))_kg_{ij}-(E(V))_jg_{ik}}{n-1}.
\end{equation}
Direct contraction gives
\begin{equation}\label{eq:case-AS-norm}
 \frac{n-1}{2n}|\Ac_E(V)|^2=\frac{n-1}{n}|E|^2|V|^2-|E(V)|^2.
\end{equation}
Indeed, in a $g$-orthonormal frame the required contractions are
\begin{align}
\sum_{i,j,k}(E_{ij}V_k-E_{ik}V_j)^2
   &=2|E|^2|V|^2-2|E(V)|^2,\notag\\
\sum_{i,j,k}\bigl((E(V))_kg_{ij}-(E(V))_jg_{ik}\bigr)^2
   &=2(n-1)|E(V)|^2,\notag\\
\sum_{i,j,k}(E_{ij}V_k-E_{ik}V_j)
 \bigl((E(V))_kg_{ij}-(E(V))_jg_{ik}\bigr)
   &=-2|E(V)|^2.
\label{eq:algebraic-contractions}
\end{align}
Expanding \eqref{eq:case-AS} using \eqref{eq:algebraic-contractions} proves
\eqref{eq:case-AS-norm}.

\subsection{Curvature conventions}
In indexed formulas, $R_{ijkl}$, $W_{ijkl}$, $C_{ijk}$, and $B_{jk}$ denote
the Riemann, Weyl, Cotton, and Bach tensors of $g$. Our curvature convention is
\begin{equation}\label{eq:curvature-convention}
 R_{ijkl}=W_{ijkl}+(A_g)_{ik}g_{jl}-(A_g)_{il}g_{jk}
 -(A_g)_{jk}g_{il}+(A_g)_{jl}g_{ik},
\end{equation}
with the Weyl tensor retained throughout for a general Einstein background.
The corresponding commutator on one-forms is
\begin{equation}\label{eq:commutator}
 (\nabla_i\nabla_j-\nabla_j\nabla_i)\xi_k=R_{ijk}{}^{\ell}\xi_\ell.
\end{equation}
The Cotton and Bach tensors are
\begin{equation}\label{eq:Cotton-Bach}
 C_{ijk}=(\nabla_g)_i(A_g)_{jk}-(\nabla_g)_j(A_g)_{ik},\qquad
 B_{jk}=(\nabla_g)^iC_{ijk}+W_{jikl}(A_g)^{il}.
\end{equation}
Let $\sigma_3(A_g)$ denote the third elementary symmetric function
of the eigenvalues of $g^{-1}A_g$. With these conventions, the sixth-order
curvature is
\begin{align}
Q_{6,g}={}&48\sigma_3(A_g)+\frac{16}{n-4}\ip{A_g}{B}
 -\frac{n+2}{2}\Delta_g(J^2)+4\Delta_g|A_g|^2
 +8\Div_g(A_g(\nabla_gJ))\notag\\
&+\Delta_g^2J-\frac{n-6}{2}J\Delta_gJ
 -4(n-6)J|A_g|^2+\frac{(n-6)(n+6)}4J^3.
\label{eq:Q6-local}
\end{align}
This is the formula of Chen--Hou~\cite[(1-1)]{ChenHou} in the convention
\eqref{eq:normalization}. Their Cotton tensor is $C^{\mathrm{CH}}_{ijk}=C_{kji}$;
the Bach contractions agree after this index permutation and the Weyl
symmetries.
The conformal covariance law is
\begin{equation}\label{eq:covariance}
 P_{6,u^2g_0}(v)=u^{-\frac{n+6}{2}}P_{6,g_0}\bigl(u^{\frac{n-6}{2}}v\bigr).
\end{equation}

In the critical dimension the curvature transformation law is
\begin{equation}\label{eq:critical-covariance}
 P_{6,g_0}(\log u)+Q_{6,g_0}=Q_{6,g}u^6,\qquad n=6.
\end{equation}

\section{Positivity derived from the sixth-order equation}\label{sec:positivity}

Throughout this section, $(M^n,g_0)$ is closed, $n\ge6$, and
\eqref{eq:background} holds. We prove positivity of $J$ and of the quantity
$\Theta_u$ in \eqref{eq:Theta}.

\subsection{The sign of the constant curvature}
The Einstein factorization~\cite[(1-4)]{ChenHou} and its fourth-order
counterpart are
\begin{equation}\label{eq:factorization}
 P_{6,g_0}=L_1L_2L_3,\qquad P_{4,g_0}=L_1L_2,
\end{equation}
where
\begin{align}
L_1&=-\Delta_{g_0}+\frac{n(n-2)}4,\qquad
 L_2=-\Delta_{g_0}+\frac{(n-4)(n+2)}4,\notag\\
L_3&=-\Delta_{g_0}+\frac{(n-6)(n+4)}4.
\label{eq:factors}
\end{align}
In the normalization \eqref{eq:background},
\[
 Q_{6,g_0}=\frac{n(n^2-4)(n^2-16)}{32}>0.
\]
\begin{lemma}\label{lem:constant-Q-positive}
If $Q_{6,g}$ is constant, then $Q_{6,g}>0$.
\end{lemma}
\begin{proof}
For $n\ge7$, \eqref{eq:covariance}, \eqref{eq:normalization},
and \eqref{eq:factorization} give
\begin{equation}\label{eq:factor-equation}
 L_1L_2L_3\bigl(u^{\frac{n-6}{2}}\bigr)=\frac{n-6}{2}Q_{6,g} u^{\frac{n+6}{2}}.
\end{equation}
Integrating \eqref{eq:factor-equation} and using the constants
in \eqref{eq:factors} and \eqref{eq:reference} yields
\[
 Q_{6,g}=\frac{n(n^2-4)(n^2-16)}{32}
 \frac{\int_Mu^{\frac{n-6}{2}}\dvo}{\int_Mu^{\frac{n+6}{2}}\dvo}>0.
\]
For $n=6$, \eqref{eq:critical-covariance},
\eqref{eq:factorization}, \eqref{eq:factors}, and \eqref{eq:reference} give
\begin{equation}\label{eq:critical-equation}
 (-\Delta_{g_0})(-\Delta_{g_0}+4)(-\Delta_{g_0}+6)(\log u)+120
       =Q_{6,g} u^6.
\end{equation}
Integrating \eqref{eq:critical-equation} gives
\[
 Q_{6,g}=120\frac{\operatorname{Vol}(M,g_0)}
                         {\int_Mu^6\dvo}>0.
\]
\end{proof}

\subsection{The conformal path and the auxiliary metric}
For the path below, $J_t$ and $E_t$ are defined from $g_t$ and
$u_t$ by \eqref{eq:Schouten} and \eqref{eq:JS-from-u}. Norms and inner products in a formula
involving $g_t$ are taken with respect to $g_t$, except where another metric
is explicitly indicated. Curvature quantities of $\widehat g_t$ retain the
full metric subscript.
\begin{proposition}\label{prop:J-positive}
If $Q_{6,g}>0$ pointwise, then $J>0$, $u^2J+2+2|\nabla_g u|^2>0$, and $\Theta_u>0$.
\end{proposition}
\begin{proof}
For $n\ge7$, define
\begin{gather}
u_t=\bigl(1-t+tu^{\frac{n-6}{2}}\bigr)^{\frac{2}{n-6}},\qquad g_t=u_t^2g_0,\notag\\
v_t=\frac{2}{n-6}L_3\bigl(u_t^{\frac{n-6}{2}}\bigr),\qquad 0\le t\le1.
\label{eq:positive-path}
\end{gather}
By \eqref{eq:positive-path}, \eqref{eq:factorization},
\eqref{eq:covariance}, \eqref{eq:normalization}, and \eqref{eq:reference},
linearity gives
\begin{equation}\label{eq:positive-forcing}
 L_1L_2v_t=(1-t)Q_{6,g_0}+tQ_{6,g}u^{\frac{n+6}{2}}>0.
\end{equation}
The zeroth-order coefficients of $L_1$ and $L_2$ are positive. Applying
the minimum principle first to $L_2v_t$ and then to $v_t$ gives $v_t>0$.
Applying \eqref{eq:Laplacian-transform} to $u_t^{\frac{n-6}{2}}$ gives
\begin{equation}\label{eq:path-Laplacian}
 \frac{2\Delta_{g_0}\bigl(u_t^{\frac{n-6}{2}}\bigr)}{(n-6)u_t^{\frac{n-2}{2}}}
       =\frac{\Delta_{g_t}u_t}{u_t}-\frac{n+4}{2}|\nabla_{g_t}\log u_t|^2.
\end{equation}
Substituting \eqref{eq:path-Laplacian} and \eqref{eq:scale-Hessian} into
\eqref{eq:positive-path}, with \eqref{eq:factors}, gives
\begin{equation}\label{eq:auxiliary-factor}
v_t=u_t^{\frac{n-2}{2}}
 \left(J_t+\frac{2(1+|\nabla_{g_t}u_t|^2)}{u_t^2}\right)>0.
\end{equation}

For $n=6$, set instead
\begin{equation}\label{eq:critical-path}
 u_t=u^t,\qquad g_t=u^{2t}g_0,
 \qquad v_t=5-t\Delta_{g_0}(\log u).
\end{equation}
The critical law \eqref{eq:critical-covariance}, together with
\eqref{eq:factorization}, \eqref{eq:factors}, \eqref{eq:reference}, and
\eqref{eq:critical-path}, gives
\begin{equation}\label{eq:critical-positive-forcing}
 (-\Delta_{g_0}+4)(-\Delta_{g_0}+6)v_t
   =(1-t)120+tQ_{6,g}u^6>0.
\end{equation}
Applying the minimum principle to \eqref{eq:critical-positive-forcing}
first for $L_2v_t$ and then for $v_t$ gives $v_t>0$.
The trace of \eqref{eq:Schouten-transform} and \eqref{eq:critical-path} give
\begin{align}
u_t^2J_t&=3-t\Delta_{g_0}(\log u)-2t^2|\nabla_{g_0}(\log u)|_{g_0}^2,\notag\\
2(1+|\nabla_{g_t}u_t|^2)&=2+2t^2|\nabla_{g_0}(\log u)|_{g_0}^2.
\label{eq:critical-path-trace}
\end{align}
Adding the two equations in \eqref{eq:critical-path-trace} gives
$v_t=u_t^2J_t+2(1+|\nabla_{g_t}u_t|^2)$ by \eqref{eq:critical-path};
this is \eqref{eq:auxiliary-factor} when $n=6$.

In either case consider
\begin{equation}\label{eq:auxiliary}
 \widehat g_t=v_t^{\frac{4}{n-4}}g_0.
\end{equation}
The fourth-order covariance law~\cite[(1.5)]{CaseGover}, with
\eqref{eq:factorization}, \eqref{eq:positive-forcing}, and
\eqref{eq:critical-positive-forcing}, gives
\begin{equation}\label{eq:auxiliary-Q4}
 P_{4,g_0}v_t=\frac{n-4}{2}Q_{4,\widehat g_t}
                    v_t^{\frac{n+4}{n-4}}>0.
\end{equation}
Thus $Q_{4,\widehat g_t}>0$. By \eqref{eq:positive-path} and
\eqref{eq:factors}, or by \eqref{eq:critical-path},
$v_0=\frac{n+4}{2}$ is constant. Hence \eqref{eq:background} and
\eqref{eq:auxiliary} give $J_{\widehat g_0}>0$.
The metrics depend smoothly on $t$; compactness ensures that their minimum
scalar curvatures vary continuously.
If this positivity first failed, then at a zero minimum of
$J_{\widehat g_t}$, the formula \eqref{eq:Q4} applied to $\widehat g_t$
would give $Q_{4,\widehat g_t}\le0$, contradicting \eqref{eq:auxiliary-Q4}.
Hence $J_{\widehat g_t}>0$ for all $t$.

By \eqref{eq:auxiliary-factor} and \eqref{eq:auxiliary}, we can write
the auxiliary metric relative to $g_t$ as
\begin{equation}\label{eq:auxiliary-relative}
 \widehat g_t=
 \left[u_tJ_t+\frac{2(1+|\nabla_{g_t}u_t|^2)}{u_t}\right]^{\frac{4}{n-4}}g_t.
\end{equation}
The product rule, \eqref{eq:scale-Hessian}, \eqref{eq:u-scale-gradient},
and \eqref{eq:u-scale-laplacian} give
\begin{align}
 &\frac1{u_t}\nabla_{g_t}
       \left[u_tJ_t+\frac{2(1+|\nabla_{g_t}u_t|^2)}{u_t}\right]\label{eq:auxiliary-derivatives}\\*
 &\qquad=\nabla_{g_t}J_t-4E_t(\nabla_{g_t}\log u_t)
                +\frac{n-4}{n}J_t\nabla_{g_t}\log u_t,\notag\\
 &\frac1{u_t}\Delta_{g_t}
       \left[u_tJ_t+\frac{2(1+|\nabla_{g_t}u_t|^2)}{u_t}\right]\notag\\*
 &\qquad=\Delta_{g_t}J_t+4|E_t|^2
        -2\ip{\nabla_{g_t}\log u_t}{\nabla_{g_t}J_t}
        +\frac{n-4}{n}J_t\frac{\Delta_{g_t}u_t}{u_t}.\notag
\end{align}
Taking the trace of \eqref{eq:Schouten-transform} for the conformal
factor in \eqref{eq:auxiliary-relative} gives
\begin{align}
&\left[u_tJ_t+\frac{2(1+|\nabla_{g_t}u_t|^2)}{u_t}\right]^{\frac{4}{n-4}}
                  J_{\widehat g_t}\notag\\
&\quad=J_t-\frac2{n-4}
  \frac{\Delta_{g_t}\left[u_tJ_t+\frac{2(1+|\nabla_{g_t}u_t|^2)}{u_t}\right]}
       {u_tJ_t+\frac{2(1+|\nabla_{g_t}u_t|^2)}{u_t}}\notag\\
&\qquad\quad-\frac4{(n-4)^2}
  \left|\nabla_{g_t}\log\left[u_tJ_t+
                      \frac{2(1+|\nabla_{g_t}u_t|^2)}{u_t}\right]\right|^2.
\label{eq:scalar-path-transform}
\end{align}
At a first zero minimum of $J_t$, one has $\nabla_{g_t}J_t=0$ and
$\Delta_{g_t}J_t\ge0$. The second identity in
\eqref{eq:auxiliary-derivatives} makes the Laplacian in
\eqref{eq:scalar-path-transform} nonnegative. Its right-hand side is then
nonpositive, contradicting $J_{\widehat g_t}>0$. Thus $J_t>0$ throughout
the path. Multiplying \eqref{eq:scalar-path-transform} by
\[
 \frac{n-4}{2}\left(J_t+
                   \frac{2(1+|\nabla_{g_t}u_t|^2)}{u_t^2}\right)
\]
and using \eqref{eq:auxiliary-derivatives} and \eqref{eq:Theta} yields
\begin{align}
\Theta_{u_t}={}&\frac{n-4}{2}
  \left(J_t+\frac{2(1+|\nabla_{g_t}u_t|^2)}{u_t^2}\right)\notag\\
&\times\left[u_tJ_t+\frac{2(1+|\nabla_{g_t}u_t|^2)}{u_t}\right]^{\frac{4}{n-4}}
             J_{\widehat g_t}>0.
\label{eq:Theta-geometry}
\end{align}
Taking $t=1$ completes the proof.
\end{proof}
The auxiliary path follows the positivity mechanism in
Gursky--Malchiodi~\cite{GurskyMalchiodi}. For constant $Q_{6,g}$,
\begin{equation}\label{eq:positivity-summary}
 \Lambda>0,\qquad J>0,\qquad
 u^2J+2+2|\nabla_gu|^2>0,\qquad\Theta_u>0.
\end{equation}

\section{Construction of the Obata vector field}\label{sec:construction}

Throughout this section, $u>0$ is smooth and satisfies
\eqref{eq:background}. The identities below do not require $Q_{6,g}$ to
be constant or any sign condition on $J$ or $\Theta_u$.

\subsection{The Bach-corrected Newton tensor}
Define
\begin{equation}\label{eq:Newton}
 \mathbb N=E^2-\frac{|E|^2}{2}g+\frac{B}{n-4}.
\end{equation}

\begin{lemma}\label{lem:Bach}
The Cotton and Bach tensors satisfy
\begin{gather}
C_{ijk}=-W_{ijk\ell}(\nabla_g\log u)^\ell,\notag\\
B_{jk}=(n-4)W_{jakb}(\nabla_g\log u)^a(\nabla_g\log u)^b,
 \qquad B(\nabla_g\log u,\cdot)=0.
\label{eq:CB-relations}
\end{gather}
For every smooth weight function $\psi$,
\begin{align}
\Dw\bigl(\psi\mathbb N(\uloggrad)\bigr)
 ={}&E^2(\uloggrad,\nabla_g\psi)
       -\frac{|E|^2}{2}\ip{\uloggrad}{\nabla_g\psi}\notag\\
&+\frac{n-2}{n}\psi E(\uloggrad,\nabla_gJ)
       -\frac{n-2}{2n}\psi\frac{\Delta_gu}{u}|E|^2\notag\\
&-\psi\tr_g(E^3)-\frac{\psi}{n-4}\ip{E}{B}.
\label{eq:Newton-flux}
\end{align}
\end{lemma}
\begin{proof}
By \eqref{eq:scale-Hessian}, \eqref{eq:u-scale-gradient}, and
\eqref{eq:Schouten},
\begin{equation}\label{eq:Hessian-Schouten}
 u_{ij}=-u(A_g)_{ij}+\frac{1+|\nabla_gu|^2}{2u}g_{ij},\qquad
 \nabla_i\left(\frac{1+|\nabla_gu|^2}{2u}\right)
       =-(A_g)_{i\ell}u^\ell.
\end{equation}
Differentiating the Hessian equation in \eqref{eq:Hessian-Schouten},
antisymmetrizing, and using \eqref{eq:curvature-convention}--\eqref{eq:Cotton-Bach},
we find that all Schouten terms cancel. We obtain
\[
 uC_{ijk}=-W_{ijk\ell}u^\ell.
\]
With \eqref{eq:Cotton-Bach}, the contracted second Bianchi identity is
\begin{equation}\label{eq:Weyl-divergence}
 (\nabla_g)^iW_{ijk\ell}=(n-3)C_{k\ell j}.
\end{equation}
Using \eqref{eq:Weyl-divergence}, the Cotton relation just proved,
\eqref{eq:Cotton-Bach}, and \eqref{eq:log-Hessian}, we obtain
\[
\begin{aligned}
 B_{jk}
 & =-(n-3)C_{k\ell j}(\uloggrad)^\ell
       -W_{ijk\ell}(\nabla_g)^i(\uloggrad)^\ell
       +W_{jikl}(A_g)^{il}\\
 & =-(n-3)C_{k\ell j}(\uloggrad)^\ell
       +W_{ijk\ell}(\uloggrad)^i(\uloggrad)^\ell.
\end{aligned}
\]
Substituting the Cotton relation in \eqref{eq:CB-relations} and using the
symmetries of $W$ proves the Bach formula in \eqref{eq:CB-relations}.
Its last assertion follows from antisymmetry.

Equations \eqref{eq:Cotton-Bach}, \eqref{eq:Schouten},
\eqref{eq:log-Hessian}, and \eqref{eq:CB-relations} give
\begin{align}
(\Div_g(E^2))_j&=\frac{n-2}{n}(E(\nabla_gJ))_j
                      +\frac12\nabla_j|E|^2+E^{ab}C_{ajb},\notag\\
E^{ab}C_{ajb}(\uloggrad)^j&=-\frac{\ip{E}{B}}{n-4}.
\label{eq:E2-divergence}
\end{align}
Substituting \eqref{eq:E2-divergence} into
\eqref{eq:case-tensor-general} with $T=E^2$ yields
\begin{align}
\Dw\bigl(E^2(\uloggrad)\bigr)
 ={}&\frac{n-2}{n}E(\uloggrad,\nabla_gJ)
       +\frac12\ip{\uloggrad}{\nabla_g|E|^2}\notag\\
&+\frac{\Delta_gu}{nu}|E|^2-\tr_g(E^3)
                           -\frac{\ip{E}{B}}{n-4}.
\label{eq:E2-weighted-divergence}
\end{align}
By \eqref{eq:CB-relations} and \eqref{eq:Newton},
$\mathbb N(\uloggrad)=E^2(\uloggrad)-\frac{|E|^2}{2}\uloggrad$.
Combining \eqref{eq:E2-weighted-divergence}, \eqref{eq:weighted-basic},
and the product rule proves \eqref{eq:Newton-flux}.
\end{proof}

\subsection{Scalar divergence identities}
Substituting \eqref{eq:Schouten} into \eqref{eq:Q6-local} gives
\begin{align}
Q_{6,g}={}&\Div_g\left[\nabla_g(\Delta_gJ+4|E|^2)
       +8E(\nabla_gJ)-\frac{n^2+2n-16}{n}J\nabla_gJ\right]\notag\\
&-\frac{n-6}{2}J\Delta_gJ
       +\frac{(n^2-4)(n^2-16)}{4n^2}J^3
       -\frac{4(n^2-12)}{n}J|E|^2\notag\\
&+16\tr_g(E^3)+\frac{16}{n-4}\ip{E}{B}.
\label{eq:Q6-div}
\end{align}
Here $48\sigma_3(A_g)=8J^3-24J|A_g|^2+16\tr_g(A_g^3)$.
For a smooth weight function $\psi$ and a smooth function $f$, the
weighted Green identity is
\begin{equation}\label{eq:weighted-Green}
 \Dw\left(\psi\nabla_gf-\frac f u\nabla_g(u\psi)\right)
       =\psi\Delta_gf-\frac f u\Delta_g(u\psi).
\end{equation}
It follows by expanding
$u^{-1}\Div_g(u\psi\nabla_gf-f\nabla_g(u\psi))$
in \eqref{eq:weighted-div}. Differentiating \eqref{eq:scale-Hessian}
and using \eqref{eq:u-scale-gradient} gives the first equation below.
Its divergence, with \eqref{eq:weighted-basic}, gives the second:
\begin{align}
u^{-1}\nabla_g\Delta_gu
   &=-\nabla_gJ-nE(\uloggrad)-2J\uloggrad,\notag\\
u^{-1}\Delta_g^2u
   &=-\Delta_gJ-(n+2)\ip{\uloggrad}{\nabla_gJ}
                  -2J\frac{\Delta_gu}{u}+n|E|^2.
\label{eq:u-iterated-derivatives}
\end{align}
The product rule, \eqref{eq:weighted-div}, and \eqref{eq:log-Hessian}
give
\begin{align}
&\Dw\left(\ip{\uloggrad}{\nabla_gJ}\nabla_gJ
                 -\frac12|\nabla_gJ|^2\uloggrad\right)\notag\\
&\quad=(\Delta_gJ)\ip{\uloggrad}{\nabla_gJ}
       -E(\nabla_gJ,\nabla_gJ)
       -\frac{n-2}{2n}\frac{\Delta_gu}{u}|\nabla_gJ|^2.
\label{eq:gradient-stress}
\end{align}

\subsection{The Obata identity}
Choose
\begin{equation}\label{eq:Green-weight}
 \psi=\frac{\Delta_gu}{u}-\frac{n(n-4)}{8(n-1)}
                         \left(J-\frac n2\right),\qquad
 f=\Delta_gJ+4|E|^2.
\end{equation}
The function $f$ groups the highest-order scalar terms in
\eqref{eq:Q6-div}. The first term of $\psi$ will cancel the undifferentiated
$Q_{6,g}$ contribution of the flux $-Q_{6,g}\uloggrad$; the second will
produce the scalar-curvature-defect source in \eqref{u:eq:source}.

Define the symmetric tensor
\begin{align}
\mathbb T_n={}&\frac{n-4}{n-1}\left(J-\frac n2\right)
 \Biggl\{-\left[\frac{(n-6)(n+8)}{12}J+
                   \frac{n(n-1)(1+|\nabla_gu|^2)}{8u^2}\right]E\notag\\
&\quad+\left(J-\frac n2\right)
 \left[\frac{(n^2-4n+2)(1+|\nabla_gu|^2)}{32u^2}\right.\notag\\
&\left.\hspace{2.2cm}+\left(\frac{n(n-4)}{48}
       +\frac{(n-2)(n-6)(n+8)}{72n}\right)
                         \left(J+\frac n4\right)\right]g\Biggr\}.
\label{u:eq:tensor}
\end{align}
Define the vector field in grouped form by
\begin{align}
 X={}&\mathbb T_n(\uloggrad)
 +\psi\nabla_gf-\frac f u\nabla_g(u\psi)
 +8\psi E(\nabla_gJ)-16\psi\mathbb N(\uloggrad)
                                          \label{u:eq:vector-field}\\
 &-\left(\frac{3n(n+5)-37}{4}
                   -\frac{3(55n-64)}{4n(n-1)}\right)J^2E(\uloggrad)
                                                               \notag\\
 &+\left[-\frac{3n^2+2n-32}{2n}J\frac{\Delta_gu}{u}
         +\frac{(n-4)^2(n+6)}{12(n-1)}J\left(J-\frac n2\right)\right]
                                                        \nabla_gJ\notag\\
 &-\frac{(n+2)(3n-4)}{2(n-1)}
   \left[\ip{\uloggrad}{\nabla_gJ}\nabla_gJ
                        -\frac12|\nabla_gJ|^2\uloggrad\right]\notag\\
 &+\left[\frac{(n^2-4)(n^2-16)}{4n^2}J^3-Q_{6,g}\right]\uloggrad.
                                                               \notag
\end{align}
Define the scalar function $F$ by
\begin{align}
 F={}&\frac{n^2+4n-8}{8(n-1)}(\Delta_gJ)^2
 -\frac{n^2+4n-8}{4(n-1)}(\Delta_gJ)\ip{\uloggrad}{\nabla_gJ}\label{u:eq:scalar-function}\\
 &-\frac{(n-4)(n-2)}{2(n-1)}|E|^2\Delta_gJ
       +\frac{(n-4)(n-2)}{2(n-1)}E(\nabla_gJ,\nabla_gJ)\notag\\
 &+\frac{n(n-4)}{48(n-1)}\left(J-\frac n2\right)
    \left[\frac{6\Delta_gu}{u}-(5n-14)J\right]\Delta_gJ\notag\\
 &+\Biggl[\frac{(n-2)(n^2+4n-8)}{4n(n-1)}\frac{\Delta_gu}{u}
       +\left(\frac{n^2+35n-33}{24}+\frac{45}{8(n-1)}-\frac{16}{n}\right)J\notag\\*
 &\hspace{4cm}+\frac{n^2(n-4)(n+2)}{48(n-1)}\Biggr]|\nabla_gJ|^2\notag\\
 &+\left[4n|E|^2+\frac{(n-4)^2(n+6)}{12(n-1)}
                        J\left(J-\frac n2\right)\right]\ip{\uloggrad}{\nabla_gJ}\notag\\
 &-\frac{2(3n^2-8n+8)}{n-1}E^2(\uloggrad,\nabla_gJ)\notag\\
 &-\left[\left(\frac{n(13n+14)}{2(n-1)}-\frac{96}{n}\right)J
       +\frac{16(n-2)}n\frac{\Delta_gu}{u}
       -\frac{(n-4)^2}{n-1}\left(J-\frac n2\right)\right]
                                      E(\uloggrad,\nabla_gJ)\notag\\
 &+\Biggl[\left(\frac{3n(n+5)-37}{4}
                    -\frac{3(55n-64)}{4n(n-1)}\right)J^2
         +\frac{4(n^2+2n-12)}nJ\frac{\Delta_gu}{u}\notag\\*
 &\qquad+\frac{8(n-2)}n\frac{(\Delta_gu)^2}{u^2}
       -\frac{(n-4)^2}{2(n-1)}\left(J-\frac n2\right)\frac{\Delta_gu}{u}\notag\\*
 &\qquad-\frac{(n-4)(n^2-12)}{2(n-1)}J\left(J-\frac n2\right)\Biggr]|E|^2\notag\\
 &-4n|E|^4+16nE^3(\uloggrad,\uloggrad)
                    -8n|E|^2E(\uloggrad,\uloggrad)\notag\\
 &+16\left(J+\frac{n(1+|\nabla_gu|^2)}{2u^2}\right)
     \left[|E(\uloggrad)|^2-\frac{|E|^2}{2}|\uloggrad|^2\right]\notag\\
 &+\frac{(n-4)(n^2-4)(n^2-16)}{32n(n-1)}
        \left(J-\frac n2\right)\left(J^3-\frac{n^3}{8}\right)
                 +\Dw\bigl(\mathbb T_n(\uloggrad)\bigr).\notag
\end{align}

\begin{proposition}\label{prop:vector-field-identity}
The vector field \eqref{u:eq:vector-field} and scalar function
\eqref{u:eq:scalar-function} satisfy \eqref{u:eq:source}.
\end{proposition}
\begin{proof}
Equations \eqref{eq:Green-weight} and \eqref{eq:u-iterated-derivatives}
give
\begin{align}
\nabla_g\psi={}&-\frac{n^2+4n-8}{8(n-1)}\nabla_gJ
                 -nE(\uloggrad)\notag\\
&-\left(2J+\frac{\Delta_gu}{u}\right)\uloggrad.
\label{eq:psi-gradient}
\end{align}
Hence
\begin{align}
-\frac{\nabla_g(u\psi)}u
 ={}&nE(\uloggrad)+\frac{n^2+4n-8}{8(n-1)}\nabla_gJ\notag\\
&+\left[2J+\frac{n(n-4)}{8(n-1)}
                 \left(J-\frac n2\right)\right]\uloggrad.
\label{eq:Green-weight-gradient}
\end{align}
The product rule gives
\[
\begin{aligned}
 \frac{\Delta_g(u\psi)}u
 ={}&\frac{\Delta_g^2u}{u}
 -\frac{n(n-4)}{8(n-1)}\Biggl[\Delta_gJ
           +2\ip{\uloggrad}{\nabla_gJ}\\
 &\hspace{3.6cm}+\left(J-\frac n2\right)\frac{\Delta_gu}{u}\Biggr].
\end{aligned}
\]
By \eqref{eq:u-iterated-derivatives},
\begin{align}
\frac{\Delta_g(u\psi)}u
 ={}&-\frac{n^2+4n-8}{8(n-1)}\Delta_gJ
      -\frac{5n^2-8}{4(n-1)}\ip{\uloggrad}{\nabla_gJ}+n|E|^2\notag\\
&-\frac{\Delta_gu}{u}\left[2J+\frac{n(n-4)}{8(n-1)}
                                \left(J-\frac n2\right)\right].
\label{eq:Green-weight-Laplacian}
\end{align}
Using $\Div_gE=\frac{n-1}{n}\nabla_gJ$, we have
\begin{align}
\Dw\bigl(8\psi E(\nabla_gJ)\bigr)
 ={}&8\psi\ip{E}{\nabla_g^2J}
       +\frac{8(n-1)}n\psi|\nabla_gJ|^2\notag\\
&+8E(\nabla_g\psi,\nabla_gJ)
       +8\psi E(\uloggrad,\nabla_gJ).
\label{eq:EJ-flux-expanded}
\end{align}
Substitute \eqref{eq:Q6-div} into \eqref{eq:weighted-Green}.
The term $-8\psi\ip{E}{\nabla_g^2J}$ cancels with
\eqref{eq:EJ-flux-expanded}, and
$-16\psi\tr_g(E^3)-16\frac{\psi}{n-4}\ip{E}{B}$ cancels with $-16$ times
\eqref{eq:Newton-flux}. Thus
\begin{align}
&\Dw\left(\psi\nabla_gf-\frac f u\nabla_g(u\psi)
             +8\psi E(\nabla_gJ)-16\psi\mathbb N(\uloggrad)\right)\notag\\*
={}&\psi Q_{6,g}-\frac f u\Delta_g(u\psi)
      +\frac{3n^2-2n-32}{2n}\psi J\Delta_gJ
      +\frac{n^2+2n-16}{n}\psi|\nabla_gJ|^2\notag\\
&-\frac{(n^2-4)(n^2-16)}{4n^2}\psi J^3
      +\frac{4(n^2-12)}n\psi J|E|^2\notag\\
&+8E(\nabla_gJ,\nabla_g\psi)
      -\frac{8(n-4)}n\psi E(\uloggrad,\nabla_gJ)\notag\\
&-16E^2(\uloggrad,\nabla_g\psi)
      +8|E|^2\ip{\uloggrad}{\nabla_g\psi}
      +\frac{8(n-2)}n\psi\frac{\Delta_gu}{u}|E|^2.
\label{eq:Green-Newton-divergence}
\end{align}
Equations \eqref{eq:weighted-basic} and
\eqref{eq:u-iterated-derivatives} also give
\begin{align}
\Dw\bigl(J^2E(\uloggrad)\bigr)
 ={}&2JE(\uloggrad,\nabla_gJ)
       +J^2\left[\frac{n-1}n\ip{\uloggrad}{\nabla_gJ}-|E|^2\right],\notag\\
\Dw\left(J\frac{\Delta_gu}{u}\nabla_gJ\right)
 ={}&J\frac{\Delta_gu}{u}\Delta_gJ
       +\left(\frac{\Delta_gu}{u}-J\right)|\nabla_gJ|^2\notag\\*
&-nJE(\uloggrad,\nabla_gJ)-2J^2\ip{\uloggrad}{\nabla_gJ},\notag\\
\Dw\left(J\left(J-\frac n2\right)\nabla_gJ\right)
 ={}&J\left(J-\frac n2\right)\Delta_gJ
       +\left(2J-\frac n2\right)|\nabla_gJ|^2\notag\\*
&+J\left(J-\frac n2\right)\ip{\uloggrad}{\nabla_gJ},\notag\\
\Dw(J^3\uloggrad)
 ={}&3J^2\ip{\uloggrad}{\nabla_gJ}+J^3\frac{\Delta_gu}{u}.
\label{eq:lower-fluxes-expanded}
\end{align}
By \eqref{eq:weighted-basic} and \eqref{eq:Green-weight},
\begin{align}
\psi Q_{6,g}+\Dw(-Q_{6,g}\uloggrad)
 ={}&-\ip{\uloggrad}{\nabla_gQ_{6,g}}\notag\\
&-\frac{n(n-4)}{8(n-1)}\left(J-\frac n2\right)Q_{6,g}.
\label{eq:curvature-source-cancellation}
\end{align}
Substitute \eqref{eq:psi-gradient} and
\eqref{eq:Green-weight-Laplacian} into
\eqref{eq:Green-Newton-divergence} and add
\eqref{eq:lower-fluxes-expanded} and \eqref{eq:gradient-stress} with the
coefficients in \eqref{u:eq:vector-field}. The coefficients of
$(\Delta_gJ)\ip{\uloggrad}{\nabla_gJ}$,
$|E|^2\Delta_gJ$, and $E^2(\uloggrad,\nabla_gJ)$ are
\[
\begin{aligned}
 \frac{5n^2-8}{4(n-1)}-\frac{(n+2)(3n-4)}{2(n-1)}
   &=-\frac{n^2+4n-8}{4(n-1)},\\
 \frac{n^2+4n-8}{2(n-1)}-n
   &=-\frac{(n-4)(n-2)}{2(n-1)},\\
 \frac{2(n^2+4n-8)}{n-1}-8n
   &=-\frac{2(3n^2-8n+8)}{n-1}.
\end{aligned}
\]
Apart from the $J(J-\frac{n}{2})$ term, the coefficient of
$J^2\ip{\uloggrad}{\nabla_gJ}$ vanishes:
\[
\begin{aligned}
 &-\frac{n-1}{n}\left[\frac{3n(n+5)-37}{4}
                      -\frac{3(55n-64)}{4n(n-1)}\right]\\
 &\qquad+\frac{3n^2+2n-32}{n}
          +\frac{3(n^2-4)(n^2-16)}{4n^2}=0.
\end{aligned}
\]
The $J^3\frac{\Delta_gu}{u}$ terms cancel. Using \eqref{eq:reference}, the
remaining scalar source terms are
\begin{align}
&\frac{n(n-4)}{8(n-1)}\left(J-\frac n2\right)
   \left[\frac{(n^2-4)(n^2-16)}{4n^2}J^3-Q_{6,g}\right]\notag\\
={}&\frac{(n-4)(n^2-4)(n^2-16)}{32n(n-1)}
         \left(J-\frac n2\right)\left(J^3-\frac{n^3}{8}\right)\notag\\
&-\frac{n(n-4)}{8(n-1)}\left(J-\frac n2\right)(Q_{6,g}-Q_{6,g_0}).
\label{eq:background-source-shift}
\end{align}
Equations \eqref{eq:psi-gradient} and \eqref{eq:scale-Hessian} give
\[
\begin{aligned}
 &16\left(\frac{\Delta_gu}{u}+2J\right)
     \left[|E(\uloggrad)|^2-\frac{|E|^2}{2}|\uloggrad|^2\right]\\
 &\quad=16\left(J+\frac{n(1+|\nabla_gu|^2)}{2u^2}\right)
     \left[|E(\uloggrad)|^2-\frac{|E|^2}{2}|\uloggrad|^2\right].
\end{aligned}
\]
Collecting the remaining terms, including
$\Dw(\mathbb T_n(\uloggrad))$, gives \eqref{u:eq:source} with $F$ as
in \eqref{u:eq:scalar-function}.
\end{proof}

\section{Coercivity and proof of the Obata identity}\label{sec:uniform-closure}

\begin{proposition}\label{u:prop:collection}
Let $n\ge6$ be an integer, $J>0$, and $\Theta_u\ge0$. The scalar function
\eqref{u:eq:scalar-function} satisfies
\begin{align}
\frac8nF\ge{}&12|E|^2\Theta_u+\frac{|\nabla_gJ|^2}{u^2}\notag\\
&+\left(\frac14\Delta_gJ+|E|^2
       -\frac{\ip{\nabla_gu}{\nabla_gJ}}{4u}\right)^2.
\label{u:eq:collection}
\end{align}
\end{proposition}

\begin{proof}
By \eqref{eq:scale-Hessian} and \eqref{eq:u-scale-gradient},
\begin{align}
\frac{\Delta_gu}{u}&=\frac{n(1+|\nabla_gu|^2)}{2u^2}-J,\notag\\
\nabla_g\left(\frac{1+|\nabla_gu|^2}{u^2}\right)
 &=-2E(\uloggrad)
   -\left[\frac{2J}{n}+\frac{1+|\nabla_gu|^2}{u^2}\right]\uloggrad.
\label{eq:coercive-local-scale}
\end{align}
Expanding \eqref{u:eq:tensor} with \eqref{eq:weighted-basic} gives
\begin{align}
\Dw(\mathbb T_n(\uloggrad))
 ={}&\frac{n-4}{n-1}\Biggl\{
 \left(J-\frac n2\right)
 \left[\frac{(n-6)(n+8)}{12}J
       +\frac{n(n-1)(1+|\nabla_gu|^2)}{8u^2}\right]|E|^2\notag\\
&-\left[\frac{(n-6)(n+8)}{12}\left(2J-\frac n2\right)
       +\frac{n(n-1)(1+|\nabla_gu|^2)}{8u^2}\right]
                    E(\uloggrad,\nabla_gJ)\notag\\
&+\frac{n(n-1)}4\left(J-\frac n2\right)|E(\uloggrad)|^2\notag\\
&+\left(J-\frac n2\right)\left[-\frac{n^2-8n+6}{16}J
       +\frac{n(n-1)(1+|\nabla_gu|^2)}{8u^2}
       +\frac{n(n^2-4n+2)}{32}\right]\notag\\*
&\hspace{1cm}\times E(\uloggrad,\uloggrad)\notag\\
&+\left(J-\frac n2\right)
 \left[\frac{n^2-16n+96}{48}J
       -\frac{n^2(1+|\nabla_gu|^2)}{16u^2}\right]
                       \ip{\uloggrad}{\nabla_gJ}\notag\\
&-\frac{n^2-4n+2}{32}\left(J-\frac n2\right)^2
 \left[\frac{2J}{n}+\frac{1+|\nabla_gu|^2}{u^2}\right]|\uloggrad|^2\notag\\
&+\left(J-\frac n2\right)^2
 \Biggl[\frac{(n^2-4n+2)(1+|\nabla_gu|^2)}{32u^2}\notag\\*
&\qquad+\left(\frac{n(n-4)}{48}
       +\frac{(n-2)(n-6)(n+8)}{72n}\right)\left(J+\frac n4\right)\Biggr]\notag\\*
&\hspace{1cm}\times
     \left[\frac{n(1+|\nabla_gu|^2)}{2u^2}-J\right]\Biggr\}.
\label{eq:coercive-tensor-divergence}
\end{align}

\par\smallskip\noindent\textit{Square completion.}
First retain the following part of \eqref{u:eq:scalar-function}:
\begin{equation}\label{eq:coercive-retained}
 \frac{(n+2)(3n-4)}{2(n-1)}|E|^2\Theta_u
 +\frac{n^2+4n-8}{8(n-1)u^2}
 \left[(1+|\nabla_gu|^2)|\nabla_gJ|^2
                   -\ip{\nabla_gu}{\nabla_gJ}^2\right].
\end{equation}
Substitute \eqref{eq:coercive-tensor-divergence} into \eqref{u:eq:scalar-function},
use \eqref{eq:coercive-local-scale}, and subtract
\eqref{eq:coercive-retained}. Leaving the rational norm term from
\eqref{eq:Theta} unexpanded, we obtain the following coefficients for
the remaining polynomial terms:
\[
\begin{array}{c@{\qquad}l}
 \text{term}&\text{coefficient}\\[5pt]
 |E|^2\Delta_gJ
 &-\dfrac{(n-4)(n-2)}{2(n-1)}
   +\dfrac{(n+2)(3n-4)}{2(n-1)}
   =\dfrac{n^2+4n-8}{n-1}\\[7pt]
 |E|^4
 &-4n+\dfrac{2(n+2)(3n-4)}{n-1}
   =\dfrac{2(n^2+4n-8)}{n-1}\\[7pt]
 |E|^2\ip{\uloggrad}{\nabla_gJ}
 &4n-\dfrac{(n+2)(3n-4)}{n-1}
   =\dfrac{(n-4)(n-2)}{n-1}.
\end{array}
\]
The first two rows, together with the coefficients of $(\Delta_gJ)^2$,
$(\Delta_gJ)\ip{\uloggrad}{\nabla_gJ}$ and
$\ip{\uloggrad}{\nabla_gJ}^2$, match the quadratic term in
\eqref{eq:scalar-quadratic-before-completion}, where
$f=\Delta_gJ+4|E|^2$. This term contributes
$-\frac{n^2+4n-8}{n-1}$ to the third row. The tensor cross term must
therefore supply the remaining coefficient $2n$, since
\[
 \frac{(n-4)(n-2)}{n-1}
       +\frac{n^2+4n-8}{n-1}=2n.
\]
The tensor squares are constructed from the norm identity
\eqref{eq:case-AS-norm} and its polarization:
\begin{align*}
 |\Ac_E(V)|^2+\frac{2n}{n-1}|E(V)|^2
   &=2|E|^2|V|^2,\\
 \ip{\Ac_E(V)}{\Ac_E(W)}
   &=2|E|^2\ip VW-\frac{2n}{n-1}\ip{E(V)}{E(W)}.
\end{align*}
Use $\nabla_gJ-4E(\uloggrad)+\frac{n-4}{n}J\uloggrad$
as the common unshifted argument. With the prefactor below, the
unshifted terms cancel \eqref{eq:tensor-denominator-balance}.
Shift the first argument by
\[
 \frac{n-4}{3}\left[J+\frac{2(1+|\nabla_gu|^2)}{u^2}\right]\uloggrad.
\]
The common factor $J+\frac{2(1+|\nabla_gu|^2)}{u^2}$ cancels the
denominator, giving the required cross coefficient
$\frac{3n}{2(n-4)}\cdot4\cdot\frac{n-4}{3}=2n$.
The coefficient of $J\uloggrad$ becomes
$\frac{n-4}{n}+\frac{n-4}{3}=\frac{(n-4)(n+3)}{3n}$.

For the second square, the combined mixed contraction has prefactor
$\frac{2n^2}{n-1}$, and the vector paired with the image of the
unshifted argument under $E$ is chosen to be
\[
 \frac{(n-4)(n-2)}{4n^2}\nabla_gJ
                 -\frac{2(n-1)}nE(\uloggrad).
\]
These coefficients match the $E(\nabla_gJ,\nabla_gJ)$ and
$E^3(\uloggrad,\uloggrad)$ terms in \eqref{u:eq:scalar-function}, respectively:
\[
 \frac{2n^2}{n-1}\frac{(n-4)(n-2)}{4n^2}
       =\frac{(n-4)(n-2)}{2(n-1)},\qquad
 \frac{2n^2}{n-1}(-4)\left(-\frac{2(n-1)}n\right)=16n.
\]
After factoring $\frac{n-4}{3}
\left[J+\frac{2(1+|\nabla_gu|^2)}{u^2}\right]$ from the shifts,
the first square already contributes $-E(\uloggrad)$ by polarization.
The second shift must therefore have coefficient
$-\frac{n-2}{n}$ on $E(\uloggrad)$; it is
\begin{align*}
 &\frac{(n-4)(n-2)}{12n^2}
       \left[J+\frac{2(1+|\nabla_gu|^2)}{u^2}\right]\\*
 &\qquad\times\left[(n-4)\nabla_gJ-4nE(\uloggrad)\right].
\end{align*}
The $E^2(\uloggrad,\nabla_gJ)$ coefficient then agrees as well,
as expanded in \eqref{eq:cubic-contraction-expanded}.
These shifts give the tensor squares
\begin{align}
 &\frac{3nu^2}{2(n-4)(u^2J+2+2|\nabla_gu|^2)}\notag\\*
 &\quad\times\Biggl[\left|\Ac_E\left(\begin{aligned}
  &\nabla_gJ-4E(\uloggrad)\\
  &+\frac{n-4}{3}\left[\frac{n+3}{n}J+
         \frac{2(1+|\nabla_gu|^2)}{u^2}\right]\uloggrad
                    \end{aligned}\right)\right|^2\notag\\
 &\qquad+\frac{2n}{n-1}\left|\begin{aligned}
  &E\left(\nabla_gJ-4E(\uloggrad)+\frac{n-4}{n}J\uloggrad\right)\\
  &+\frac{(n-4)(n-2)}{12n^2}
       \left[J+\frac{2(1+|\nabla_gu|^2)}{u^2}\right]\\
  &\hspace{1cm}\times\left[(n-4)\nabla_gJ-4nE(\uloggrad)\right]
                    \end{aligned}\right|^2\Biggr].
                    \label{eq:coercive-two-squares}
\end{align}
The remaining norm term is
\begin{equation}\label{eq:coercive-unused}
 \frac{(5n-8)u^2|E|^2}{(n-1)(n-4)(u^2J+2+2|\nabla_gu|^2)}
 \left|\nabla_gJ-4E(\uloggrad)+\frac{n-4}{n}J\uloggrad\right|^2\ge0.
\end{equation}
The rational term in \eqref{eq:coercive-retained} combines with
\eqref{eq:coercive-unused} as
\begin{align}
&\frac{[-(n+2)(3n-4)+(5n-8)]u^2|E|^2}
 {(n-1)(n-4)(u^2J+2+2|\nabla_gu|^2)}\notag\\
&\quad\times
 \left|\nabla_gJ-4E(\uloggrad)+\frac{n-4}{n}J\uloggrad\right|^2\notag\\
={}&-\frac{3nu^2|E|^2}{(n-4)(u^2J+2+2|\nabla_gu|^2)}\notag\\
&\quad\times
 \left|\nabla_gJ-4E(\uloggrad)+\frac{n-4}{n}J\uloggrad\right|^2.
\label{eq:tensor-denominator-balance}
\end{align}
Adding \eqref{eq:tensor-denominator-balance} to
\eqref{eq:coercive-two-squares} and using \eqref{eq:case-AS-norm} gives
\begin{align}
&2n|E|^2\ip{\uloggrad}{\nabla_gJ}
       -8n|E|^2E(\uloggrad,\uloggrad)\notag\\
&+\frac{n-4}{3}\left[(n+6)J+
                  \frac{2n(1+|\nabla_gu|^2)}{u^2}\right]
                                    |E|^2|\uloggrad|^2\notag\\
&+\frac{2n^2}{n-1}\Biggl\langle
       E(\nabla_gJ)-4E^2(\uloggrad)
                 +\frac{n-4}{n}JE(\uloggrad),\notag\\*
&\hspace{2.5cm}\frac{(n-4)(n-2)}{4n^2}\nabla_gJ
                      -\frac{2(n-1)}nE(\uloggrad)\Biggr\rangle\notag\\
&+\frac{n^2(n-4)}{3(n-1)}
 \left[J+\frac{2(1+|\nabla_gu|^2)}{u^2}\right]
 \Biggl[\frac{(n-4)^2(n-2)^2}{16n^4}|\nabla_gJ|^2\notag\\*
&\qquad-\frac{(n-4)(n-2)^2}{2n^3}E(\uloggrad,\nabla_gJ)
                  -\frac{4(n-1)}{n^2}|E(\uloggrad)|^2\Biggr].
\label{eq:tensor-completion-expanded}
\end{align}
The contraction in \eqref{eq:tensor-completion-expanded} expands as
\begin{align}
&\frac{2n^2}{n-1}\Biggl\langle
       E(\nabla_gJ)-4E^2(\uloggrad)
                 +\frac{n-4}{n}JE(\uloggrad),\notag\\*
&\hspace{2.5cm}\frac{(n-4)(n-2)}{4n^2}\nabla_gJ
                      -\frac{2(n-1)}nE(\uloggrad)\Biggr\rangle\notag\\
={}&\frac{(n-4)(n-2)}{2(n-1)}E(\nabla_gJ,\nabla_gJ)\notag\\
&-\frac{2(3n^2-8n+8)}{n-1}E^2(\uloggrad,\nabla_gJ)
       +16nE^3(\uloggrad,\uloggrad)\notag\\
&+\frac{(n-4)^2(n-2)}{2n(n-1)}JE(\uloggrad,\nabla_gJ)
       -4(n-4)J|E(\uloggrad)|^2.
\label{eq:cubic-contraction-expanded}
\end{align}

To collect the remaining terms, rewrite \eqref{eq:case-AS-norm} as
\begin{equation}\label{eq:new-absorption}
 \frac{1+|\nabla_gu|^2}{u^2}|E|^2
  =\frac{|E|^2}{u^2}+\frac12|\Ac_E(\uloggrad)|^2
                     +\frac{n}{n-1}|E(\uloggrad)|^2.
\end{equation}
The pure $J^2|E|^2$, $J(\frac{n}{2})|E|^2$, and $(\frac{n}{2})^2|E|^2$
remainders cancel. The other contributions are collected below.

\par\smallskip\noindent\textit{The trace-free scalar terms.}
These are
\begin{align}
&\frac{n(n+2)}3\left[J+\frac{2(1+|\nabla_gu|^2)}{u^2}\right]\frac{|E|^2}{u^2}\notag\\
&\quad+\frac{n(n-4)}{24(n-1)}\left[
   (13n+17)J+\frac{32(n-1)(1+|\nabla_gu|^2)}{u^2}
       +\frac{3n(3n-7)}2\right]\notag\\
&\hspace{2.5cm}\times\left[\frac{|E|^2}{u^2}
                           +\frac12|\Ac_E(\uloggrad)|^2\right]\ge0.
\label{eq:coercive-trace-free}
\end{align}

\par\smallskip\noindent\textit{The scalar completion.}
The $\Delta_gJ$ terms after subtraction of
\eqref{eq:coercive-retained} agree with those of
\begin{align}
&\frac{n^2+4n-8}{8(n-1)}
              \left(f-\ip{\uloggrad}{\nabla_gJ}\right)^2\notag\\
&+\frac{n(n-4)}{48(n-1)}\left(J-\frac n2\right)
       \left[\frac{3n(1+|\nabla_gu|^2)}{u^2}-(5n-8)J\right]
              \left(f-\ip{\uloggrad}{\nabla_gJ}\right),
\label{eq:scalar-quadratic-before-completion}
\end{align}
where $f$ is defined in \eqref{eq:Green-weight}. We use
\begin{align}
\frac{n^2+4n-8}{8(n-1)}
     &=\frac n{128}+\frac{15n^2+65n-128}{128(n-1)},\notag\\
6\frac{\Delta_gu}{u}-(5n-14)J&=3n\frac{1+|\nabla_gu|^2}{u^2}-(5n-8)J.
\label{eq:scalar-coefficient-splitting}
\end{align}
We retain
\begin{equation}\label{eq:reserved-square}
 \frac n{128}\left(\Delta_gJ+4|E|^2
                  -\frac{\ip{\nabla_gu}{\nabla_gJ}}u\right)^2.
\end{equation}
The remaining quadratic coefficient is
$\frac{15n^2+65n-128}{128(n-1)}>0$.
The remaining scalar square is
\begin{align}
&\frac{15n^2+65n-128}{128(n-1)}
 \Biggl[\Delta_gJ+4|E|^2-\ip{\uloggrad}{\nabla_gJ}\notag\\
&\qquad+\frac{4n(n-4)(J-\frac{n}{2})}{3(15n^2+65n-128)}
    \left(\frac{3n(1+|\nabla_gu|^2)}{u^2}-(5n-8)J\right)\Biggr]^2.
\label{eq:remaining-scalar-square}
\end{align}
The completion is justified by the exact scalar identity
\begin{align}
&\frac{n^2+4n-8}{8(n-1)}
              \left(f-\ip{\uloggrad}{\nabla_gJ}\right)^2\notag\\
&+\frac{n(n-4)(J-\frac{n}{2})}{48(n-1)}
       \left[\frac{3n(1+|\nabla_gu|^2)}{u^2}-(5n-8)J\right]
              \left(f-\ip{\uloggrad}{\nabla_gJ}\right)\notag\\
={}&\frac n{128}\left(f-\ip{\uloggrad}{\nabla_gJ}\right)^2\notag\\
&+\frac{15n^2+65n-128}{128(n-1)}
  \Biggl[f-\ip{\uloggrad}{\nabla_gJ}\notag\\*
&\qquad+\frac{4n(n-4)(J-\frac{n}{2})}{3(15n^2+65n-128)}
       \left[\frac{3n(1+|\nabla_gu|^2)}{u^2}-(5n-8)J\right]\Biggr]^2\notag\\
&-\frac{n^2(n-4)^2(J-\frac{n}{2})^2}{72(n-1)(15n^2+65n-128)}
       \left[\frac{3n(1+|\nabla_gu|^2)}{u^2}-(5n-8)J\right]^2.
\label{eq:scalar-completion-exact}
\end{align}
Thus \eqref{eq:reserved-square} and \eqref{eq:remaining-scalar-square}
are retained. Combining the last term of
\eqref{eq:scalar-completion-exact} with the other scalar-defect terms
leaves $\frac{(n-4)(J-\frac{n}{2})^2}{n-1}$ times the sum of the following polynomial
and \eqref{eq:coercive-scalar-other}:
\begin{align}
&\left[\frac{(n^2-4)(n^2-16)}{32n}
 -\frac{n(n-4)}{48}-\frac{(n-2)(n-6)(n+8)}{72n}\right.\notag\\
&\hspace{1.1cm}\left.
 -\frac{n^2(n-4)(5n-8)^2}{72(15n^2+65n-128)}\right]J^2\notag\\
&+\left[\frac{n^3(n-4)(5n-8)}{12(15n^2+65n-128)}
                 -\frac{n^2-4n+2}{32}\right]J\frac{1+|\nabla_gu|^2}{u^2}\notag\\
&+\frac{(n-6)(9n^3+49n^2+126n+860)+5544}{576}J\notag\\
&+\frac{(n-6)\{n^2(7n^2+34n-214)+267n-66\}+116}
       {64(15n^2+65n-128)}\frac{(1+|\nabla_gu|^2)^2}{u^4}.
\label{eq:coercive-scalar-coefficients}
\end{align}
The other contributions to the factored remainder are
\begin{align}
&\frac{n\{(n-6)(5n^2+9n-14)+90\}}{1152}
                 \frac{1+|\nabla_gu|^2}{u^2}
 +\frac{n(n^2-4)(n^2-16)}{128}\notag\\
&\quad+\frac1{u^2}\Biggl[
    \left(\frac{n^2(n-4)}{96}
         +\frac{(n-2)(n-6)(n+8)}{144}\right)J\notag\\
&\hspace{2.5cm}+\frac{(3n^2-9n+4)(1+|\nabla_gu|^2)}{64u^2}
         +\frac{n(n^2-4n+2)}{128}\Biggr],
\label{eq:coercive-scalar-other}
\end{align}
and are positive. For $n\ge6$, the elementary bounds
$15n^2\le15n^2+65n-128\le24n^2$ and $3n\le5n-8\le5n$ show that
the first two coefficients in \eqref{eq:coercive-scalar-coefficients} are
bounded below, respectively, by
\begin{equation}\label{eq:coercive-scalar-bounds}
 \frac{(n-6)(7n^3+92n^2+84n+1128)+7344}{864n}>0,\qquad
 \frac{(n-6)(n^2-n+6)+30}{96}>0.
\end{equation}
The third coefficient is positive as displayed. The fourth is positive
because $7n^2+34n-214>0$ and $267n-66>0$ for $n\ge6$.
Together with \eqref{eq:coercive-scalar-other}, this proves that
\eqref{eq:reserved-square} is retained after the scalar completion
\eqref{eq:remaining-scalar-square}.

\par\smallskip\noindent\textit{The terms weighted by $J$.}
After removing $J$, these terms are
\begin{align}
&\left[\frac14+\frac{23(n-4)}{144}
           +\frac{7(n-4)^2}{144n(n-1)}\right]|3\nabla_gJ-8E(\uloggrad)|^2\notag\\
&\quad+\frac1{32}\left|(n-4)E(\uloggrad)
       -\frac{n^2-8n+6}{n-1}(J-\frac{n}{2})\uloggrad\right|^2\notag\\
&\quad+\left[\frac{73(n-4)^2}{96}
     +\frac{5(n-4)\{n(n-3)+20\}}{72(n-1)^2}
     +\frac{112(n-4)}{9n(n-1)}\right]|E(\uloggrad)|^2\notag\\
&\quad+\left[\frac{(n-4)^2}{48}+\frac1{54}+\frac3{2(n-1)}
     +(n-6)\left\{\frac14+\frac{17n+48}{54n^2}\right\}\right]|\nabla_gJ|^2\notag\\
&\quad+\Biggl[\frac{7(n-4)}{4n(n-1)}
     +\frac{(n-4)(n-6)\{n(5n^2+9n-23)+90\}}{288n(n-1)}\notag\\*
&\hspace{2.5cm}+\frac{(3n-2)\{(n-6)(2n-1)+4\}}{32(n-1)^2}\Biggr]
                  (J-\frac{n}{2})^2|\uloggrad|^2.
\label{eq:coercive-J-terms}
\end{align}
All coefficients are nonnegative for $n\ge6$.

\par\smallskip\noindent\textit{The terms weighted by $\frac{n(1+|\nabla_gu|^2)}{2u^2}$.}
First complete the square
\begin{equation}\label{eq:fixed-scale-square}
 \frac{n-4}{32}
       \left|(J-\frac{n}{2})\uloggrad+4E(\uloggrad)\right|^2.
\end{equation}
The remaining quadratic form is
\begin{align}
&\left[\frac n6+\frac74+\frac{15}{4(n-1)}
       -\frac{46}{3n}+\frac{16}{n^2}-\frac{64}{3n^3}\right]
                                                      |\nabla_gJ|^2\notag\\
&+2\left[\frac{5n}{24}-\frac{67}{6}+\frac3{n-1}
                       +\frac{80}{3n}-\frac{64}{3n^2}\right]
                                  E(\uloggrad,\nabla_gJ)\notag\\
&+\left[\frac{13n}{6}+\frac{46}{3}-\frac8{n-1}-\frac{64}{3n}\right]
                                                      |E(\uloggrad)|^2.
\label{eq:coercive-matrix-coefficients}
\end{align}
\needspace{11\baselineskip}
The coefficient matrix in \eqref{eq:coercive-matrix-coefficients}
is positive definite. At $n=6$, its leading principal minors are
\begin{equation}\label{eq:coercive-matrix-six}
 \frac{209}{162}>0,\qquad \frac{1231}{32400}>0.
\end{equation}
For real $n\ge6$, direct differentiation gives diagonal entries of
the derivative matrix greater than $\frac16$ and $\frac{13}{6}$,
and an off-diagonal entry of absolute value less than $\frac12$.
Its determinant is therefore greater than
\begin{equation}\label{eq:coercive-matrix-derivative}
 \frac16\frac{13}{6}-\left(\frac12\right)^2=\frac19>0.
\end{equation}
The derivative matrix is positive definite, so integration from $6$ to $n$
proves the assertion.

\par\smallskip\noindent\textit{The terms weighted by $\frac{n}{2}$.}
After removing the factor $\frac{n}{2}$, these terms are
\begin{align}
&\frac{n-4}{96(n-1)}\Biggl[
   \frac32(n^2-4n+2)|(J-\frac{n}{2})\uloggrad+2E(\uloggrad)|^2\notag\\
&\qquad+8n|2\nabla_gJ-E(\uloggrad)|^2
         +4n(n-6)|\nabla_gJ+E(\uloggrad)|^2\notag\\
&\qquad+\left\{2(n-6)(n+14)+\frac{12(9n-13)}{n-1}\right\}
                  |E(\uloggrad)|^2\Biggr],
\label{eq:coercive-constant-terms}
\end{align}
which are nonnegative for all $n\ge6$.

Collecting the terms in \eqref{u:eq:scalar-function} using
\eqref{eq:coercive-tensor-divergence},
\eqref{eq:tensor-completion-expanded},
\eqref{eq:cubic-contraction-expanded},
\eqref{eq:scalar-completion-exact}, and \eqref{eq:new-absorption},
we obtain the exact decomposition below. Here $[\eqref{eq:coercive-retained}]$
denotes the expression in the referenced display, with any inequality
omitted, and the same convention applies to the other brackets:
\begin{align*}
F={}&[\eqref{eq:coercive-retained}]+[\eqref{eq:reserved-square}]\\
&+[\eqref{eq:coercive-two-squares}]+[\eqref{eq:coercive-unused}]
  +[\eqref{eq:coercive-trace-free}]+[\eqref{eq:remaining-scalar-square}]\\
&+\frac{(n-4)(J-\frac{n}{2})^2}{n-1}
  \bigl([\eqref{eq:coercive-scalar-coefficients}]
       +[\eqref{eq:coercive-scalar-other}]\bigr)\\
&+J[\eqref{eq:coercive-J-terms}]\\
&+\frac{n(1+|\nabla_gu|^2)}{2u^2}
  \bigl([\eqref{eq:fixed-scale-square}]
       +[\eqref{eq:coercive-matrix-coefficients}]\bigr)
  +\frac n2[\eqref{eq:coercive-constant-terms}].
\end{align*}
The first line contains the terms to be retained. Each expression in the
remaining lines is nonnegative by the preceding arguments, and all its
weights are nonnegative for $n\ge6$ and $J>0$. Discarding these lines gives
\begin{align}
F\ge{}&\frac{(n+2)(3n-4)}{2(n-1)}|E|^2\Theta_u\notag\\*
&+\frac{n^2+4n-8}{8(n-1)u^2}
 \left[(1+|\nabla_gu|^2)|\nabla_gJ|^2
                   -\ip{\nabla_gu}{\nabla_gJ}^2\right]\notag\\*
&+\frac n{128}\left(\Delta_gJ+4|E|^2
                   -\ip{\uloggrad}{\nabla_gJ}\right)^2.
\label{eq:coercive-exact-sum}
\end{align}

Finally, Cauchy--Schwarz gives
\begin{equation}\label{eq:coercive-Cauchy-Schwarz}
 (1+|\nabla_gu|^2)|\nabla_gJ|^2
                -\ip{\nabla_gu}{\nabla_gJ}^2\ge|\nabla_gJ|^2.
\end{equation}
The coefficient differences in \eqref{eq:coercive-retained} are
\begin{equation}\label{eq:coercive-coefficient-differences}
 \frac{(n+2)(3n-4)}{2(n-1)}-\frac{3n}2
       =\frac{5n-8}{2(n-1)}>0,\qquad
 \frac{n^2+4n-8}{8(n-1)}-\frac n8
       =\frac{5n-8}{8(n-1)}>0.
\end{equation}
Since $|E|^2\Theta_u\ge0$, \eqref{eq:coercive-Cauchy-Schwarz} and
\eqref{eq:coercive-coefficient-differences} applied to
\eqref{eq:coercive-exact-sum} give
\begin{align*}
F\ge{}&\frac{3n}{2}|E|^2\Theta_u+\frac n{8u^2}|\nabla_gJ|^2\\
&+\frac n{128}\left(\Delta_gJ+4|E|^2
                  -\frac{\ip{\nabla_gu}{\nabla_gJ}}u\right)^2.
\end{align*}
Multiplication by $\frac{8}{n}$ proves \eqref{u:eq:collection}.
\end{proof}

\begin{proof}[Completion of the proof of Theorem~\ref{thm:obata}]
Discard the scalar square in \eqref{u:eq:collection} to obtain
\eqref{u:eq:coercivity}. The local source identity \eqref{u:eq:source}
was proved in Section~\ref{sec:construction}. Multiplication by $u$ and
integration give \eqref{eq:obata-integral}.
\end{proof}

\section*{Declaration}
The local rigidity result for the sixth-order $Q$-curvature equation,
proved using an Obata-type identity, was established in the first author's
doctoral dissertation (2024). In preparing the present manuscript, the authors
used ChatGPT (OpenAI) to simplify and check the complicated algebraic terms
arising from the Cotton and Bach tensors, primarily to verify the
computations in the proof of Proposition~\ref{u:prop:collection}.
The authors also used ChatGPT to assist with the typesetting of the manuscript.
The authors take full responsibility for the mathematical arguments and
the final manuscript.


\begin{thebibliography}{99}
\bibitem{Case}
J.~S. Case, \emph{The Obata--V\'etois argument and its applications},
J. Reine Angew. Math. \textbf{815} (2024), 23--40.
\href{https://doi.org/10.1515/crelle-2024-0048}{doi:10.1515/crelle-2024-0048}.

\bibitem{CaseGover}
J.~S. Case and A.~R. Gover, \emph{The GJMS operators in geometry, analysis and physics},
J. London Math. Soc. (2) \textbf{113} (2026), no.~1, e70375, 21~pp.
\href{https://doi.org/10.1112/jlms.70375}{doi:10.1112/jlms.70375}.

\bibitem{ChangYang}
S.-Y.~A. Chang and P.~C. Yang, \emph{On uniqueness of solutions of $n$-th order
differential equations in conformal geometry}, Math. Res. Lett. \textbf{4}
(1997), no.~1, 91--102.
\href{https://doi.org/10.4310/MRL.1997.v4.n1.a9}{doi:10.4310/MRL.1997.v4.n1.a9}.

\bibitem{ChenHou}
X.~Chen and F.~Hou, \emph{Remarks on GJMS operator of order six},
Pacific J. Math. \textbf{289} (2017), no.~1, 35--70.
\href{https://doi.org/10.2140/pjm.2017.289.35}{doi:10.2140/pjm.2017.289.35}.

\bibitem{GJMS}
C.~R. Graham, R.~Jenne, L.~J. Mason, and G.~A.~J. Sparling,
\emph{Conformally invariant powers of the Laplacian, I: Existence},
J. London Math. Soc. (2) \textbf{46} (1992), no.~3, 557--565.
\href{https://doi.org/10.1112/jlms/s2-46.3.557}{doi:10.1112/jlms/s2-46.3.557}.

\bibitem{Gursky1997}
M.~J. Gursky, \emph{Uniqueness of the functional determinant},
Comm. Math. Phys. \textbf{189} (1997), no.~3, 655--665.
\href{https://doi.org/10.1007/s002200050223}{doi:10.1007/s002200050223}.

\bibitem{GurskyMalchiodi}
M.~J. Gursky and A.~Malchiodi, \emph{A strong maximum principle for the
Paneitz operator and a non-local flow for the $Q$-curvature},
J. Eur. Math. Soc. \textbf{17} (2015), no.~9, 2137--2173.
\href{https://doi.org/10.4171/JEMS/553}{doi:10.4171/JEMS/553}.

\bibitem{HyderMartinazzi}
A.~Hyder and L.~Martinazzi, \emph{Conformal metrics on $\mathbb R^{2m}$ with
constant $Q$-curvature, prescribed volume and asymptotic behavior},
Discrete Contin. Dyn. Syst. \textbf{35} (2015), no.~1, 283--299.
\href{https://doi.org/10.3934/dcds.2015.35.283}{doi:10.3934/dcds.2015.35.283}.

\bibitem{LiWei}
M.~Li and J.~Wei, \emph{A remark on the Case--Gursky--V\'etois identity
and its applications}, Proc. Amer. Math. Soc. \textbf{153} (2025),
no.~8, 3417--3430.
\href{https://doi.org/10.1090/proc/17255}{doi:10.1090/proc/17255}.

\bibitem{MartinazziClassification}
L.~Martinazzi, \emph{Classification of solutions to the higher order
Liouville's equation on $\mathbb R^{2m}$}, Math. Z. \textbf{263} (2009),
no.~2, 307--329.
\href{https://doi.org/10.1007/s00209-008-0419-1}{doi:10.1007/s00209-008-0419-1}.

\bibitem{MartinazziVolume}
L.~Martinazzi, \emph{Conformal metrics on $\mathbb R^{2m}$ with constant
$Q$-curvature and large volume}, Ann. Inst. H. Poincar\'e C Anal.
Non Lin\'eaire \textbf{30} (2013), no.~6, 969--982.
\href{https://doi.org/10.1016/j.anihpc.2012.12.007}{doi:10.1016/j.anihpc.2012.12.007}.

\bibitem{Obata1971}
M.~Obata, \emph{The conjectures on conformal transformations of
Riemannian manifolds}, J. Differential Geom. \textbf{6} (1971),
no.~2, 247--258.
\href{https://doi.org/10.4310/jdg/1214430407}{doi:10.4310/jdg/1214430407}.

\bibitem{Vetois}
J.~V\'etois, \emph{Uniqueness of conformal metrics with constant
$Q$-curvature on closed Einstein manifolds}, Potential Anal.
\textbf{61} (2024), no.~3, 485--500.
\href{https://doi.org/10.1007/s11118-023-10117-1}{doi:10.1007/s11118-023-10117-1}.

\bibitem{WeiXu}
J.~Wei and X.~Xu, \emph{Classification of solutions of higher order
conformally invariant equations}, Math. Ann. \textbf{313} (1999), no.~2,
207--228.
\href{https://doi.org/10.1007/s002080050258}{doi:10.1007/s002080050258}.
\end{thebibliography}
\end{document}